\documentclass[12pt]{article}

\usepackage{a4wide}
\usepackage{amssymb}
\usepackage{amsfonts}
\usepackage{amsmath}
\usepackage{amsthm, bm}

\newtheorem{theorem}{Theorem}[section]
\newtheorem{lemma}[theorem]{Lemma}
\newtheorem{conjecture}[theorem]{Conjecture}
\newtheorem{proposition}[theorem]{Proposition}
\newtheorem{definition}[theorem]{Definition}
\newtheorem{example}[theorem]{Example}
\newtheorem{remark}[theorem]{Remark}

\newcommand{\re}{\mathbb{R}}

\newcommand{\vc}[1]{{\mathbf{#1}}}
\newcommand{\Aalpha}[1]{\alpha}

\begin{document}

\makeatletter
\title{A proof of the Shepp--Olkin entropy  \\ monotonicity conjecture}
\author{Erwan Hillion\thanks{Aix Marseille Univ, CNRS, Centrale Marseille, I2M, Marseille, France } \and Oliver Johnson\thanks{School of Mathematics, University of Bristol, University Walk, Bristol, BS8 1TW, UK. Email {\tt maotj@bristol.ac.uk}}
}
\date{\today}
\maketitle
\makeatother

\begin{abstract} \noindent Consider tossing a collection of coins, each fair or biased towards heads, and take the distribution of the total number of heads that result. It is natural to conjecture that this distribution should be `more random' when each coin is fairer. Indeed, Shepp and Olkin conjectured  that the Shannon entropy of this distribution is monotonically increasing in this case. We resolve this conjecture, by proving that this intuition is correct. Our proof uses a construction which was previously developed by the authors to prove a related conjecture of Shepp and Olkin concerning concavity of entropy. We discuss whether this result can be generalized to $q$-R\'{e}nyi and $q$-Tsallis entropies, for a range of values of $q$.
\end{abstract}


\section{Introduction and notation}

In this paper, we consider the entropy of Poisson--binomial random variables (sums of independent Bernoulli random variables). 
Given parameters $\vc{p} = (p_1, \ldots, p_n)$ (where $0 \leq p_i \leq 1$)
we will write $f_{\vc{p}}$ for the probability mass function of the random variable
$B_1 + \ldots + B_n$, where $B_i$ are independent with $B_i \sim$ Bernoulli$(p_i)$. We can write the Shannon entropy as a function of the parameters:
\begin{equation} \label{eq:entdef}  H( \vc{p}) = - \sum_{k=0}^n f_{\vc{p}}(k) \log f_{\vc{p}}(k). \end{equation}
Shepp and Olkin \cite{shepp} made the following conjecture ``on the basis of numerical calculations and verification in the special cases
$n=2,3$'':  
\begin{conjecture}[Shepp--Olkin monotonicity conjecture] \label{conj:somon} If
 all $p_i \leq 1/2$ then $H(\vc{p})$ is a non-decreasing function of $\vc{p}$. \end{conjecture}

 The main contribution of the present paper
is to prove that this conjecture is correct, and to give conditions for equality (see Theorem \ref{thm:somon}). The heart of our argument will be a representation of one probability
distribution in terms of another, via so-called mixing coefficients $\alpha$ 
(see Definition \ref{def:alphadef}) previously introduced in \cite{johnson34} in a form motivated by optimal 
transport. The key property of these coefficients is that their differences are decreasing functions (see Proposition \ref{prop:dec}). We  analyse the resulting expressions using  elementary tools; specifically the interplay between the product rule \eqref{eq:NablaProduct} for discrete differentiation and
an integration by parts formula, Lemma \ref{lem:intparts}. It may be surprising that only such simple tools are required to resolve this conjecture, however we remark that the interplay between product rules and integration by parts lies behind much of the power of the Bakry-\'{E}mery $\Gamma_2$-calculus for continuous random variables (as alluded to in the 
dedication of \cite{bakry2}). This may suggest that the Shepp--Olkin conjecture should be viewed in the context of an emerging discrete Bakry-\'{E}mery theory (see also \cite{caputo, johnson38}).

Given the naturalness of Conjecture \ref{conj:somon}, and its simplicity of statement, it is perhaps surprising that it has remained
open for over 40 years (although \cite{shepp} was published as a book chapter in 1981, the conjecture was first formulated in the
corresponding technical report \cite{shepp2} in July 1978). Indeed, we are not even aware of any published work that resolves special cases.

In the same paper \cite{shepp},  Shepp and Olkin also conjectured that the entropy $H(\vc{p})$ is concave in $\vc{p} \in [0,1]^n$, which was proved in \cite{johnson34} and \cite{johnson36}, using constructions based on optimal transport which we will describe below.  Similarly to the monotonicity conjecture, little published work had previously addressed the entropy concavity conjecture, though limited progress in special cases had been made in \cite{johnsonc6} and \cite{hillion2}.

In their original paper \cite{shepp}, Shepp and Olkin did prove some related results. In particular, they showed \cite[Theorem 1]{shepp} that the entropy
$H(\vc{p})$ is Schur concave in $\vc{p}$, and hence deduced a maximum entropy property for binomial random variables (see also contemporary work of Mateev \cite{mateev}, as well as later extensions by \cite{harremoes, yu}). Further, they showed \cite[Theorem 2]{shepp} that the entropy of $H(\vc{p})$ is concave in a single argument $p_i$ (see \eqref{eq:concave} below)  and \cite[Theorem 4]{shepp} that the entropy of the 
binomial distribution is concave. The Schur concavity property of Shepp and Olkin was generalized to a wider range of functionals
and parametric families including negative binomials by Karlin and Rinott \cite{karlin}.

However, none of these results appear to be particularly relevant to the monotonicity conjecture, Conjecture \ref{conj:somon}.
We can reformulate the conjecture to say that if $p_i \leq q_i \leq 1/2$ for each $i = 1, 2, \ldots n$ then $H(\vc{p}) \leq H(\vc{q})$. Clearly, by symmetry of the arguments,  it is sufficient to verify this in the case where $p_i = q_i$ for $i = 1, 2, \ldots n-1$, which in turn means
that  it is enough to check that $\frac{\partial}{\partial p_n} H(\vc{p}) \geq 0$. Without loss of generality we will assume throughout that all $p_i$ are non-zero, so that $f(k)$ is non-zero for all $k$ in $\{ 0, 1, \ldots, n \}$.

In this case, following standard calculations in \cite{shepp}, we can take $\vc{p}(t) = (p_1, \ldots, p_{n-1}, p_n(t))$  where $p_n(t) = p_n + t$,  omit the subscript on $f_{\vc{p}(t)}$ for brevity and write
\begin{equation}
 \frac{\partial f}{\partial t}(k)  = g(k-1) - g(k), \mbox{ \;\;\; for $k = 0, 1, \ldots, n$,} \label{eq:fder}
\end{equation}
where $g$ is the probability mass function of $B_1 + \ldots + B_{n-1}$, which is supported on the set $\{0, \ldots, n-1 \}$ and does not depend on $t$. Here and throughout we take $g(-1) = g(n) = 0$ if necessary.

As in \cite[Theorem 2]{shepp} we can use \eqref{eq:fder} to evaluate the first two derivatives of $H(\vc{p}(t))$ as a function of $t$. Direct substitution gives
\begin{align}
\frac{\partial H}{\partial t} & = \sum_{k=0}^n \left( g(k) - g(k-1) \right) \log f(k), \label{eq:todeal} \\
\frac{\partial^2 H}{\partial t^2} & =  - \sum_{k=0}^n \frac{ \left( g(k) - g(k-1) \right)^2}{ f(k)}. \label{eq:concave}
\end{align} 
The negativity of each term in \eqref{eq:concave} tells us directly that (as proved in \cite[Theorem 2]{shepp}) the entropy $H(\vc{p}(t))$ is concave in $t$ (of course, we also know this from the full Shepp--Olkin theorem proved in \cite{johnson34, johnson36}) so it is sufficient to prove that the derivative $\frac{\partial H}{\partial t}$ is non-negative in the case $p_n = 1/2$, since the derivative is therefore larger for any smaller values of $p_n$.

However, at this stage, further progress is elusive. Considering convolution with $B_n$ means that we can express $f(k) = (g(k) + g(k-1))/2$. However, substituting this in \eqref{eq:todeal} does not suggest an obvious way forward in general, though it is possible to use the resulting formula to resolve certain special cases. For example, careful cancellation in the case where $p_1 = p_2 = \ldots = p_{n-1} = 1/2$
and hence $g$ is binomial allows us to deduce that, in this case, the entropy derivative \eqref{eq:todeal} equals zero (see Example \ref{ex:binomial} below for an alternative view of this). However, this calculation does not give any
particular insight into why the binomial case might be extreme in the sense of the conjecture. 

Instead of expressing $f$ as a linear combination of $g$, our key observation  is that  we can express $g$ as a weighted linear combination of $f$, as described in the following section.

\section{Entropy derivative and mixing coefficients}

The following construction and  notation were introduced in \cite{johnson34}, based on the `hypergeometric thinning' construction of Yu \cite[Definition 2.2]{yu}. The key is to observe that in general we can write
\begin{equation}
g(k)  = \alpha_{k+1} f(k+1) + (1-\alpha_k) f(k), \label{eq:lincomb} 
\end{equation}
for certain `mixing coefficients' $ (\alpha)_{k=0, 1, \ldots, n}$. The general construction for  $(\alpha)_{k=0, 1, \ldots, n}$ 
in the case of Shepp--Olkin paths  is given in \cite[Proposition 5.1]{johnson34}, but
in the specific case where only $p_n$ varies, in the case $p_n = 1/2$, we can simply
define the following values:
\begin{definition} \label{def:alphadef} For $k = 0, \ldots, n$, define
\begin{equation} \label{eq:alphadef}
\alpha_{k} :=  \frac{g(k-1)}{2 f(k)} = \frac{g(k-1)}{g(k-1) + g(k)}. \end{equation}
\end{definition}
 Here, by by direct calculation using \eqref{eq:alphadef} and the fact that $g(k)^2 > g(k-1) g(k+1)$
(see \cite{niculescu}), the mixing coefficients satisfy
\begin{equation} \label{eq:alphamon}
0 = \alpha_0 < \alpha_1 <  \alpha_2 < \ldots < \alpha_n = 1. \end{equation}
In \cite[Proposition 5.2]{johnson34} this result was stated in the form $\alpha_{k-1} \leq \alpha_k$, but the strict inequalities
will help us to resolve the case of equality in Conjecture \ref{conj:somon}.
It will often be useful for us to observe that Definition \ref{def:alphadef} implies that for $k= 0, 1, \ldots, n$
\begin{equation}
g(k)  = 2 \alpha_{k+1} f(k+1) = 2 (1-\alpha_k) f(k), \label{eq:alphavals}
\end{equation}
and that for $k= 0, 1, \ldots, n-1$
\begin{equation}
\alpha_{k+1} g(k+1) = (1-\alpha_{k+1}) g(k). \label{eq:alphavals2}
\end{equation}

\begin{remark} Summing \eqref{eq:alphavals}, we can directly calculate that
\begin{equation} \label{eq:alphamean} \sum_{k=0}^n f(k) \alpha_k = \frac{1}{2},\end{equation}
which will play an important role in our proof of Conjecture \ref{conj:somon}. Further, it is interesting to note by rearranging \eqref{eq:alphadef} that $\alpha_k \leq 1/2$ if and only if $g(k-1) \leq g(k)$, which
by the unimodality of $g$ (see for example \cite{niculescu}) means that $k \leq {\rm mode}(g)$. This may suggest that the Shepp-Olkin conjecture can be understood as relating
to the skewness of the random variables $g$. Direct calcuation shows that the centred third moment of $B_1 + \ldots + B_{n-1} $ is $\sum_{i=1}^{n-1} p_i (1-p_i)(1-2p_i) \geq 0$, but it is not immediately clear how this positive skew will affect the entropy of $f$.
\end{remark}

\begin{remark}
In \cite{johnson34} we used these mixing coefficients $(\alpha)_{k=0, 1, \ldots, n}$ to formulate a discrete analogue of the Benamou--Brenier formula
\cite{benamou} from optimal transport theory, which gave an understanding of certain interpolation paths of discrete probability measures
(including Shepp--Olkin paths) as geodesics in a metric space. We do not require this interpretation here, but simply
study the properties of $\alpha_k$ in their own right.
\end{remark}
We now define a function which will form the basis of our proof of Conjecture \ref{conj:somon}:
\begin{definition} \label{def:varphisign} 
Define
$$
\psi(\alpha) := \alpha \log \alpha - (1-\alpha) \log (1-\alpha) - (2 - 2 \log 2)(\alpha - 1/2),
$$
where we take $0 \log 0 = 0$ to ensure that $\psi$ is continuous at $0$ and at $1$.
\end{definition}

\begin{remark} \label{rem:psiprops} \mbox{ }
\begin{enumerate}
\item
Note that  we can express 
\begin{equation} \label{eq:powerseries}
\psi(\alpha) = - \sum_{r=1}^\infty \frac{(2 \alpha - 1)^{2r+1}}{2r(2r+1)}, \end{equation}
as a power series with only odd terms  with all coefficients negative. Further, comparison with $\sum_{r=1}^\infty
\frac{1}{2r(2r+1)} = 1 - \log 2 < \infty$ shows that this series converges absolutely for all $\alpha \in [0,1]$.
\item Hence the $\psi$ is non-increasing and antisymmetric about $1/2$ (with 
$\psi(\alpha) = - \psi(1-\alpha)$) and with $\psi(0)  = 1 - \log 2 > 0$, $\psi(1/2) = 0$ and $\psi(1) = -(1 - \log 2) < 0$.
\end{enumerate}
\end{remark}

We can express the derivative of entropy in terms of these functions and the mixing coefficients, as follows:

\begin{proposition} \label{prop:key}
\begin{equation} \label{eq:key} \frac{\partial H}{\partial t}  = 2 \sum_{k=0}^n f(k) \psi(\alpha_k). \end{equation}
Hence by \eqref{eq:powerseries}, the entropy derivative \eqref{eq:todeal} is positive  if each of the odd centred moments $ \sum_{k=0}^n f(k)  (\alpha_k - 1/2)^{2r+1} \leq 0$,
for $r = 1,2, \ldots$.
\end{proposition}
\begin{proof}
Using the fact that $g(-1) = g(n) = 0$ and adding cancelling terms into the sum \eqref{eq:todeal}, we can use 
$g(k)/(2 f(k)) = 1- \alpha_k$ and $g(k-1)/(2 f(k)) = \alpha_k$ (see
 \eqref{eq:alphavals}) to obtain that
\begin{align}
\frac{\partial H}{\partial t} & = \sum_{k=0}^{n-1} g(k)  \log f(k) - \sum_{k=1}^n  g(k-1)  \log f(k) \notag \\
& = \sum_{k=0}^{n-1} g(k)  \log \left( \frac{ f(k)}{g(k)/2} \right) - \sum_{k=1}^n g(k-1) \log \left( \frac{ f(k)}{g(k-1)/2} \right) \ \notag \\
& = 2 \sum_{k=0}^{n-1} f(k) \left( -  (1-\alpha_k) \log (1-\alpha_k) \right)  +  2  \sum_{k=1}^n f(k) \left( \alpha_{k} \log \alpha_{k} \right) \notag \\
& = 2 \sum_{k=0}^n f(k)  \left( \alpha_{k} \log \alpha_{k} -  (1-\alpha_k) \log (1-\alpha_k) \right) \label{eq:derivative}
\end{align}
since $\alpha_0 = 0$ and $\alpha_n = 1$ so that $\alpha_0 \log \alpha_0 = (1- \alpha_n) \log (1-\alpha_n) = 0$.

Since (by \eqref{eq:alphamean} above) the $\sum_{k=0}^n f(k) \alpha_k = 1/2$, subtracting off the linear term makes no difference to the sum and we can rewrite \eqref{eq:derivative} as 
$ 2 \sum_{k=0}^n f(k) \psi(\alpha_k)$, as required.

We can exchange the order of summation in 
$$ 2 \sum_{k=0}^n f(k) \psi(\alpha_k) = - 2 \sum_{k=0}^n  f(k) \sum_{r=1}^\infty \frac{(2 \alpha_k - 1)^{2r+1}}{2r(2r+1)}
= -  \sum_{r=1}^\infty \frac{2^{2r+1}}{r(2r+1)} \sum_{k=0}^n  f(k) ( \alpha_k - 1/2)^{2r+1},$$
because of Fubini's theorem, since as mentioned above the power series for $\psi$ converges absolutely. Hence
if each odd centered moment is negative then the entropy derivative \eqref{eq:todeal} is positive.
\end{proof}

Notice that, using  Remark \ref {rem:psiprops}  and Proposition \ref{prop:key}, we can immediately deduce that the entropy derivative must lie between $-2(1-\log 2)$ and $2(1-\log 2)$.

\begin{example} \label{ex:binomial}
In the case where $p_1 = p_2 = \ldots = p_n$, since $g$ is Binomial$(n-1,1/2)$, we know that $g(k)/g(k-1) = (n-k)/k$ for
$k= 1, \ldots, n$, so that 
$$ \alpha_k = \frac{1}{1 + g(k)/g(k-1)} = \frac{1}{1 + (n-k)/k} = \frac{k}{n}$$
Using this, and the fact that $\alpha_0 = 0$, since $f$ is Binomial$(n,1/2)$ we know that $f(k) = f(n-k)$ and $\alpha_k - 1/2 = -(\alpha_{n-k} - 1/2)$ so  each odd centred moment satisfies
$$ \sum_{k=0}^n f(k)  (\alpha_k - 1/2)^{2r+1} = - \sum_{k=0}^n f(n-k)  (\alpha_{n-k} - 1/2)^{2r+1} = -  \sum_{l=0}^n f(l)  (\alpha_l - 1/2)^{2r+1}$$
by relabelling, and hence equals zero.
\end{example}

 We shall argue that the binomial example, Example \ref{ex:binomial}, represents the extreme case using the following property, which will be key for us:

\begin{proposition} \label{prop:dec} If all the $p_i \leq 1/2$ then
\begin{equation} \label{eq:dec} \alpha_{k+1} - \alpha_k \mbox{ \;\;\; is decreasing in $k = 0, \ldots, n-1$.} \end{equation}
\end{proposition}
\begin{proof} See Appendix \ref{sec:decpf}. \end{proof}

Note that comparing averages, and taking the values of $\alpha_0=0$ and $\alpha_n=1$ from \eqref{eq:alphamon}, Proposition \ref{prop:dec}  implies that
$$ \frac{\alpha_k}{k} = \frac{ \sum_{\ell = 0}^{k-1} ( \alpha_{\ell+1} - \alpha_\ell)}{k}
\geq \frac{ \sum_{\ell = 0}^{n-1} ( \alpha_{\ell+1} - \alpha_\ell)}{n} =\frac{\alpha_n}{n} = \frac{1}{n},$$
or that if all the $p_i \leq 1/2$ then $\alpha_k \geq k/n$ for $0 \leq k \leq n$, showing that the binomial distribution of  Example \ref{ex:binomial} is the extreme case in this sense.

\section{Proof of Shepp--Olkin monotonicity conjecture}

We are now in a position to complete our proof of Conjecture \ref{conj:somon}. 

\begin{definition} First we introduce some further notation:
\begin{enumerate}
\item We denote $\beta_k = \alpha_k -1/2$. In particular $\alpha_{k+1}+\alpha_{k-p}-1 = \beta_{k+1}+\beta_{k-p}$. As proven in \eqref{eq:alphamon} and Proposition \ref{prop:dec} respectively, the sequence $(\beta_k)_k$ is non-decreasing and $(\beta_{k+1}-\beta_k)_k$ is non-increasing.
\item We define the family $(A_p(k))_k$ by $A_{0}(k) = 1$ and, for $p \geq 1$, $A_p(k) = \prod_{j=0}^{p-1} \alpha_{k-j}$. Here if $p \geq k+1$, we take $A_p(k) = 0$ (this reflects
that the product includes the term $\alpha_0 = 0$).
\item We define the family $(B_p(k))_k$ by $B_{0}(k) = 1$ and, for $1 \leq p \leq k+1$, write $B_p(k) = \prod_{j=0}^{p-1}(\beta_{k+1}-\beta_{k-j})$. For other values of $k$, $B_p(k)$ is not defined.
\item The notation $\nabla$ stands for the left-derivative operator: $\nabla v(k) = v(k)-v(k-1)$. This operator satisfies a product rule of the form:
\begin{equation} \label{eq:NablaProduct}
\nabla(vw)(k) = v(k) \nabla w(k) + w(k-1) \nabla v(k).
\end{equation}
\item For $n \geq 1$ and $p \geq 1$ we define the polynomial (symmetric in its inputs)
\begin{equation}
Q_{n,p}(X_1,\ldots,X_p) = \sum X_1^{i_1} \ldots X_p^{i_p},
\end{equation}
where the sum is taken over all the $p$-tuples $0 \leq i_1,\ldots,i_p \leq n$ such that $i_1+\ldots+i_p = n$. We also set $Q_{0,p}=1$ and $Q_{-1,p}=0$ for all $p \geq 1$. Clearly, $Q_{n,p}(X_1,\ldots,X_p)$ is non-negative if $X_1,\ldots,X_p$ are non-negative.
\end{enumerate}
\end{definition}

We now state three technical lemmas that we will require in the proof; each of these are proved in Appendix \ref{sec:technical}. First, the fact that $B_p(k)$ is decreasing in $k$:
\begin{lemma} \label{lem:nablaBneg}
Given $p \geq 0$, we have $\nabla B_p(k) = B_p(k) - B_p(k-1) \leq 0$ for $k \geq p$.
\end{lemma}
We next  give an integration by parts formula. Note that although we restrict the range of summation for technical reasons, the values $A_p(k)$ and
$A_{p+1}(k)$ are zero outside the respective ranges:
\begin{lemma} \label{lem:intparts}
For any function $v(k)$ that is well-defined on $p \leq k \leq n-1$  and any $p \geq 0$ we have
\begin{equation} \label{eq:IPP}
\sum_{k=p}^{n-1} g(k) A_p(k) v(k) (\beta_{k+1}+\beta_{k-p})  = \sum_{k=p+1}^{n-1} g(k) A_{p+1}(k) \nabla v(k).
\end{equation}
\end{lemma}
Finally, a result concerning differences of the $Q$ polynomials:
\begin{lemma} \label{lem:QnpIdentity}
For $n \geq 0$ and $p \geq 1$, we have
\begin{equation} \label{eq:QnpIdentity}
Q_{n,p}(X_2,\ldots,X_{p+1}) - Q_{n,p}(X_1,\ldots,X_p) = (X_{p+1}-X_1) Q_{n-1,p+1}(X_1,\ldots,X_{p+1}).
\end{equation}
\end{lemma}

We now state and prove the key Proposition:

\begin{proposition} \label{prop:SrpIncreasing}
Fix $r \geq 1$. The sequence $(S_{r,p})_{1 \leq p \leq r+1}$ defined by
\begin{equation*}
S_{r,p} = \sum_{k=p-1}^{n-1} g(k) A_{p-1}(k) B_p(k) Q_{r-p,p+1}(\beta_{k+1}^2,\ldots,\beta_{k-p+1}^2)\left(\beta_{k+1}+\beta_{k-p+1} \right),
\end{equation*}
is increasing in $p$. Here note that $S_{r,r+1} = 0$ since $Q_{-1,p+1} = 0$.
\end{proposition}

(Note that by restricting to this range of summation the $B_p(k)$ is well-defined).

\begin{proof} We first take $v(k) = B_p(k) Q_{r-p,p+1}(\beta_{k+1}^2,\ldots,\beta_{k-p+1}^2)$ in the integration by parts formula \eqref{eq:IPP} from Lemma \ref{lem:intparts} to write
\begin{equation*}
S_{r,p} = \sum_{k=p}^{n-1} g(k) A_p(k) \nabla \left[B_p(k) Q_{r-p,p+1}(\beta_{k+1}^2,\ldots,\beta_{k-p+1}^2)\right],
\end{equation*}
where here and throughout the proof, the $\nabla$ refers to a difference in the $k$ parameter.
Now, using the product rule \eqref{eq:NablaProduct} with $v(k) = B_p(k)$ and $w(k) = Q_{r-p,p+1}(\beta_{k+1}^2,\ldots,\beta_{k-p+1}^2)$ yields:
\begin{eqnarray*}
S_{r,p} &=&  \sum_{k=p}^{n-1} g(k) A_p(k) B_{p}(k) \nabla \left[Q_{r-p,p+1}(\beta_{k+1}^2,\ldots,\beta_{k-p+1}^2)\right] \\
&& + \sum_{k=p}^{n-1} g(k) A_p(k)  Q_{r-p,p+1}(\beta_{k}^2,\ldots,\beta_{k-p}^2) \nabla B_p(k) 
\end{eqnarray*}
 In order to transform the first sum, we use equation~\eqref{eq:QnpIdentity} from Lemma \ref{lem:QnpIdentity}:
\begin{eqnarray*}
\lefteqn{ \nabla \left[Q_{r-p,p+1}(\beta_{k+1}^2,\ldots,\beta_{k-p+1}^2)\right] } \\
& = & Q_{r-p,p+1}(\beta_{k+1}^2,\ldots,\beta_{k-p+1}^2) - Q_{r-p,p+1}(\beta_{k}^2,\ldots,\beta_{k-p}^2) \\
 &=& (\beta_{k+1}^2-\beta_{k-p}^2) Q_{r-p-1,p+2}(\beta_k^2,\ldots,\beta_{k-p+1}^2) \\
&=&(\beta_{k+1}-\beta_{k-p}) Q_{r-p-1,p+2}(\beta_k^2,\ldots,\beta_{k-p+1}^2)(\beta_{k+1}+\beta_{k-p}).
\end{eqnarray*}
Further Lemma \ref{lem:nablaBneg} gives that $ \nabla B_p(k)  \leq 0$ for $k \geq p$ and $Q_{r-p,p+1}(\beta_{k}^2,\ldots,\beta_{k-p}^2) \geq 0$, so the second sum is $\leq 0$.
We finally conclude that
\begin{equation*}
S_{r,p} \leq \sum_{k=p}^{n-1} g(k) A_p(k) B_{p+1}(k) Q_{r-p-1,p+2}(\beta_{k+1}^2,\ldots,\beta_{k-p+1}^2)\left(\beta_{k+1}+\beta_{k-p} \right) = S_{r,p+1}, 
\end{equation*}
as we have $B_p(k)(\beta_{k+1}-\beta_{k-p}) = B_{p+1}(k)$.
\end{proof}

We are now able to prove the following theorem, which confirms Conjecture \ref{conj:somon}:
\begin{theorem} \label{thm:somon} If
 all $p_i \leq 1/2$ then $H(\vc{p})$ is a non-decreasing function of $\vc{p}$.
Equality holds if and only if each $p_i$ equals $0$ or $\frac{1}{2}$. \end{theorem}
\begin{proof}
As described in Proposition \ref{prop:key}, it is sufficient for us to prove that for every $r \geq 1$ we have
\begin{equation} \label{eq:oddpower}
\sum_{k=0}^n f(k)  (\alpha_k - 1/2)^{2r+1} = \sum_{k=0}^n f(k) \beta_k^{2r+1} \leq 0.
\end{equation}
Using \eqref{eq:alphavals} we know that $\alpha_k = g(k-1)/(2 f(k))$ and $1-\alpha_k = g(k)/(2 f(k))$, so that (subtracting these two expressions) 
$$ \beta_k = \alpha_k - 1/2 = \frac{ (g(k-1) - g(k)}{4 f(k)}.$$
This means that, using a standard factorization of $\left(\beta_{k+1}^{2r}-\beta_k^{2r} \right) $, since $g(-1) = g(n) = 0$
we can write
\begin{eqnarray}
4 \sum_{k=0}^n f(k) \beta_k^{2r+1}  &=& \sum_{k=0}^n (g(k-1)-g(k)) \beta_k^{2r} \nonumber \\
&=& \sum_{k=0}^{n-1} g(k) \left(\beta_{k+1}^{2r}-\beta_k^{2r} \right) \nonumber \\
&=& \sum_{k=0}^{n-1} g(k) \left(\beta_{k+1}-\beta_k \right)\left(\beta_{k+1}+\beta_k\right)Q_{r-1,2}(\beta_{k+1}^2,\beta_k^2) \label{eq:finalexp} \\
&=& S_{r,1}.  \nonumber 
\end{eqnarray}
However, Proposition~\ref{prop:SrpIncreasing} gives 
\begin{equation*}
S_{r,1} \leq S_{r,r+1}=0,
\end{equation*}
and we are done. Note that an examination of \eqref{eq:finalexp} allows us to deduce conditions under which equality
holds for the cubic case ($r=1$).  
In this case we can rewrite \eqref{eq:finalexp} using the integration by parts formula \eqref{eq:IPP} as
\begin{eqnarray} 
 \sum_{k=0}^{n-1} g(k) \left( \beta_{k+1} - \beta_k \right) \left( \beta_{k+1} + \beta_k \right) 
& = & \sum_{k=0}^{n-1} g(k) \alpha_{k} \left( ( \beta_{k+1} - \beta_k) - (\beta_{k} - \beta_{k-1}) \right) \label{eq:cubic}
\end{eqnarray}
Here  $g(k)$ and $\alpha_k$ are positive for $k \geq 1$, and Proposition \ref{prop:dec} tells us that the second bracket is negative,
and so the centered third moment equals zero if and only $\beta_{k+1} - \beta_k$ is constant in $k$, which means that $\alpha_k = k/n$.
However, \eqref{eq:alphamean} tells us that this implies that $$\frac{1}{2} = \sum_{k=0}^n f(k) \alpha_k
= \frac{1}{n} \sum_{k=0}^n f(k) k = \frac{1}{n} \sum_{i=1}^n p_i,$$
so that equality can hold if and only if $p_i \equiv 1/2$.
\end{proof}

\section{Monotonicity of R\'{e}nyi and Tsallis entropies}

As in \cite[Section 4]{johnson36} where a similar discussion considered the question of concavity of entropies, we briefly discuss whether Theorem \ref{thm:somon} may extend to prove that $q$-R\'{e}nyi
and $q$-Tsallis entropies are always increasing functions of $\vc{p}$ for $p_i \leq 1/2$.
We make the following definitions, each of which reduce to the Shannon entropy
\eqref{eq:entdef} as $q \rightarrow 1$.

\begin{definition} For $f_{\vc{p}}$ as defined above, for $0 \leq q \leq \infty$  define
\begin{eqnarray}
\mbox{1. $q$-R\'{e}nyi entropy \cite{renyi2}: \;\;}
\label{eq:renyi} H_{R,q}(\vc{p}) &  = & \frac{1}{1-q} \log \left( \sum_{x=0}^n f_{\vc{p}}(x)^q \right), \\
\mbox{2. $q$-Tsallis entropy  \cite{tsallis}: \;\;}
 \label{eq:tsallis} H_{T,q}(\vc{p}) & = & \frac{1}{q-1}  \left( 1-  \sum_{x=0}^n f_{\vc{p}}(x)^q \right).\end{eqnarray}
\end{definition}
Note that, unlike the concavity case of \cite[Section 4]{johnson36}, since they are both monotone functions of 
$\sum_{x=0}^n f_{\vc{p}}(x)^q$, both $H_{R,q}(\vc{p})$ and $H_{T,q}(\vc{p})$ will be increasing in $\vc{p}$ in the same cases. 
We can provide analogues of \eqref{eq:todeal} and \eqref{eq:concave} by  (for $q \neq 1$)
\begin{align}
\frac{\partial H_{T,q}}{\partial t} & = - \frac{q}{q-1} \sum_{k=0}^n \left( g(k-1) - g(k) \right) f(k)^{q-1}, \label{eq:Ttodeal} \\
\frac{\partial^2 H_{T,q}}{\partial t^2} & =  - q  \sum_{k=0}^n  \left( g(k-1) - g(k) \right)^2 f(k)^{q-2}. \label{eq:Tconcave}
\end{align} 
Again, the second term is negative, and  therefore $H_{T,q}(\vc{p})$ will be increasing for all $p_n  \leq 1/2$ if it is increasing in the case $p_n = 1/2$. Clearly for $q=0$ \eqref{eq:Ttodeal} shows that the entropy is constant (indeed we know that in this case $H_{R,q} = \log (n+1)$ and $H_{T,q} = n$).

Curiously, we can simplify \eqref{eq:Ttodeal} in the case of collision entropy ($q=2$) by substituting for $f$ as a linear combination
of $g$ (which is the argument that did not work for $q=1$).
\begin{lemma} For $q=2$, if $p_n \leq 1/2$
\begin{equation} \label{eq:T2eq} \frac{\partial H_{T,q}}{\partial t} = (1-2p_n) \sum_{k=0}^n \left( g(k-1) - g(k) \right)^2 \geq 0.
\end{equation}
\end{lemma}
\begin{proof} In \eqref{eq:Ttodeal} we obtain
\begin{align*}
\frac{\partial H_{T,q}}{\partial t} & = - 2 \sum_{k=0}^n \left( g(k-1) - g(k) \right) f(k) \\
& = - 2 \sum_{k=0}^n \left( g(k-1) - g(k) \right) \left( (1-p_n) g(k) + p_n g(k-1) \right) \\
& = 2 \sum_{k=0}^{n}( 1-p_n - p_n)  g(k)^2 - (1-p_n - p_n) g(k) g(k-1),
\end{align*}
which is equal to the term stated in \eqref{eq:T2eq} by relabelling. Note that (curiously) this property will hold for any $g$, including the mass function of any $B_1 + \ldots + B_{n-1}$ (not necessarily with $p_i < 1/2$).
\end{proof}

It may be natural to conjecture that Tsallis (and hence R\'{e}nyi) entropy is increasing for all $q$. However, the following example shows that this property in fact can fail for
$q > 2$ (note that R\'{e}nyi entropy is not concave in the same range -- see \cite[Lemma 4.3]{johnson36}). 
\begin{example}
Consider $n=2$ with $p_1 = 1/2 - \epsilon$ and $p_2 = 1/2$. Direct substitution in \eqref{eq:Ttodeal} gives that the entropy derivative is exactly
\begin{equation} - \frac{q 2^{1-q}}{q-1} \left( \left( \frac{1}{2} - \epsilon \right)^q - \left( \frac{1}{2} + \epsilon \right)^q + 2 \epsilon \right) 
= - \frac{q 2^{2-2q}}{q-1} (2^q - 2 q) \epsilon + O(\epsilon^3),\end{equation}
and we note that $2^q - 2 q \geq 0$ for $q > 2$, so the leading coefficient is negative and so the derivative will be negative for $\epsilon$ sufficiently small.
\end{example}

However, we conjecture that these entropies are increasing for $0 \leq q \leq 2$, since we know that the result holds for $q=0,1,2$:

\begin{conjecture} If
 all $p_i \leq 1/2$ then Tsallis entropy $H_{T,q}(\vc{p})$ and R\'{e}nyi entropy $H_{R ,q}(\vc{p})$ are non-decreasing functions of $\vc{p}$ for $0 \leq q \leq 2$. \end{conjecture}
We use an argument similar to that which gave Proposition \ref{prop:key} to give a moment-based condition related to this 
conjecture.

\begin{proposition}
Let us fix $0<q<2$. If, for all $r \geq 1$,
\begin{equation} \label{eq:TsallisBetaCondition}
\sum_{k=0}^{n} f(k)^q \beta_k^{2r+1} \leq 0,
\end{equation}
then $\frac{\partial H_{T,q}}{\partial t} \geq 0$ holds.
\end{proposition}

\begin{proof}
We first add a telescoping sum in equation~\eqref{eq:Ttodeal}:
\begin{eqnarray*}
\frac{\partial H_{T,q}}{\partial t} & = & - \frac{q}{q-1} \sum_{k=0}^n \left( g(k-1) - g(k) \right) f(k)^{q-1} + \frac{1}{q-1} \sum_{k=0}^n g(k-1)^q-g(k)^q \\
&=& -\frac{1}{q-1} \sum_{k=0}^n f(k)^q \left[\left(\frac{g(k)}{f(k)}\right)^q-\left(\frac{g(k-1)}{f(k)}\right)^q+q \left(\frac{g(k)}{f(k)}-\frac{g(k-1)}{f(k)} \right)\right] \\
&=& -\frac{1}{q-1} \sum_{k=0}^n f(k)^q \left[\left(2(1-\alpha_k)\right)^q-\left(2 \alpha_k \right)^q+q \left(2(1-\alpha_k)-2 \alpha_k \right)\right] \\
&=& -\frac{1}{q-1} \sum_{k=0}^n f(k)^q \left[\left(1-2\beta_k \right)^q-\left(1+2\beta_k \right)^q+4 \beta_k q \right] \\
&=& \sum_{k=0}^n f(k)^q \psi_q( 2 \beta_k),
\end{eqnarray*}
where, using the binomial theorem, the function $\psi_q$ can be expressed as
\begin{equation}
\psi_q(x) = -\frac{(1-x)^q-(1+x)^q+2qx}{q-1} = \sum_{r=1}^\infty \left ( q \prod_{i=2}^{2r}(q-i) \right) \frac{1}{(2r+1)!} x^{2r+1}. 
\end{equation}
From the assumption $0<q<2$, it follows that $q \prod_{i=2}^{2r}(q-i) < 0$. The proof is completed as in Proposition~\ref{prop:key}.
\end{proof}

\section*{Acknowledgments}

Oliver Johnson would like to thank the Isaac Newton Institute for Mathematical Sciences, Cambridge, for support and hospitality during the workshop {\em Beyond I.I.D. in information theory} where work on this paper was undertaken. This work was supported by EPSRC grant no {\tt EP/K032208/1}.

\appendix

\section{Proof of Proposition \ref{prop:dec}} \label{sec:decpf}

\begin{proof}
Using $\alpha_{k+1}= g(k)/(g(k) + g(k+1))$, we can express the difference
$$ \alpha_{k+1} - \alpha_k = \frac{ g(k)^2 - g(k+1) g(k-1)}{ (g(k+1) + g(k))(g(k) + g(k-1))} = \frac{D_g(k)}{ (g(k+1) + g(k))(g(k) + g(k-1))},$$
so the property is equivalent to
\begin{equation} \label{eq:tocontrol} (g(k+2) + g(k+1)) D_g(k)  \geq (g(k-1) + g(k)) D_g(k+1), \end{equation}
where we write $D_g(k) = g(k)^2 - g(k+1) g(k-1)$.

We write $g^{(i)}$ for the mass function of $\sum_{j \neq i} B_i$, the sum of the first $(n-1)$ Bernoulli variables with the $i$th one omitted,
and write $D^{(i)}(k) = \left( g^{(i)}(k) \right)^2 - g^{(i)}(k+1) g^{(i)}(k-1)$ and $E^{(i)}(k)
=  g^{(i)}(k) g^{(i)}(k+1) - g^{(i)}(k+2) g^{(i)}(k-1)$.

The cases $q=1$ and $q=2$ of \cite[Lemma A1]{johnson34} give that
\begin{align}
g(k) g(k-1) & = \sum_{i=1}^{n-1} p_i (1-p_i) D^{(i)}(k-1), \\
2 g(k-1) g(k+1) & = \sum_{i=1}^{n-1} p_i (1-p_i) E^{(i)}(k-1), 
\end{align}
and by direct substitution we deduce that 
\begin{eqnarray}
\lefteqn{
g(k)^2 g(k+1)  + g(k-1) g(k) g(k+2) - 2 g(k-1) g(k+1)^2 } \nonumber \\ 
& = & \sum_{i=1}^{n-1} p_i (1-p_i)
\left( g(k) D^{(i)}(k) + g(k+2)  D^{(i)}(k-1) - g(k+1) E^{(i)}(k-1) \right) \nonumber \\
& =&  \sum_{i=1}^{n-1} p_i (1-p_i) \left( (1-p_i) \frac{  \left( D^{(i)}(k) \right)^2 - D^{(i)}(k-1) D^{(i)}(k+1)}
{g^{(i)}(k)} \right), \label{eq:form1}
\end{eqnarray}
where the second equality follows directly by taking $g(k) = (1-p_i) g^{(i)}(k) + p_i g^{(i)}(k-1)$ for each $i$ and simplifying inside the bracket. Similarly, we can find 
\begin{eqnarray}
\lefteqn{
g(k) g(k+1)^2  + g(k-1) g(k+1) g(k+2) - 2 g(k)^2 g(k+2) } \nonumber \\ 
& = & \sum_{i=1}^{n-1} p_i (1-p_i)
\left( g(k+1) D^{(i)}(k) + g(k-1)  D^{(i)}(k+1) - g(k) E^{(i)}(k) \right) \nonumber \\
& =&  \sum_{i=1}^{n-1} p_i (1-p_i) \left( p_i \frac{  \left( D^{(i)}(k) \right)^2 - D^{(i)}(k-1) D^{(i)}(k+1)}
{g^{(i)}(k)} \right). \label{eq:form2}
\end{eqnarray}
Comparing \eqref{eq:form1} and \eqref{eq:form2}, using the fact that $ \left( D^{(i)}(k) \right)^2 - D^{(i)}(k-1) D^{(i)}(k+1) \geq 0$ (see \cite{branden} or 
\cite[Section 2.2]{johnson36}) we deduce that if all $p_i \leq 1/2$ then
\begin{eqnarray*}
\lefteqn{
 g(k)^2 g(k+1)  + g(k-1) g(k) g(k+2) - 2 g(k-1) g(k+1)^2 } \\
& \geq & g(k) g(k+1)^2  + g(k-1) g(k+1) g(k+2) - 2 g(k)^2 g(k+2),
\end{eqnarray*}
which by rearranging is precisely \eqref{eq:tocontrol} above. 
\end{proof}

\section{Proof of technical lemmas} \label{sec:technical}

\begin{proof}[Proof of Lemma \ref{lem:nablaBneg}] Given $j \geq 0$, each sequence $(\beta_{k+1}-\beta_{k-j})_{k}$ is non-negative and non-increasing, so any product of such sequences is also non-increasing. 
In more detail, the 
$$ \nabla B_p(k) = B_p(k) - B_p(k-1) = \prod_{j=0}^{p-1} (\beta_{k+1}-\beta_{k-j}) - \prod_{j=0}^{p-1} (\beta_{k}-\beta_{k-j-1}), $$
and for each $j \geq 0$, the $ (\beta_{k+1}-\beta_{k-j}) \leq (\beta_{k}-\beta_{k-j-1})$ because Proposition \ref{prop:dec} gives
$(\beta_{k+1}-\beta_{k}) \leq (\beta_{k-j}-\beta_{k-j-1})$, and each term in the product is positive by \eqref{eq:alphamon}. Further, each of these terms is well-defined since $k-j - 1 \geq k - p \geq 0$.
\end{proof}

\begin{proof}[Proof of Lemma \ref{lem:intparts}] Using the facts that $g(n) = 0$ and that $A_{p+1}(0) = 0$, we write the RHS of \eqref{eq:IPP} as
\begin{eqnarray*}
\lefteqn{
 \sum_{k=p+1}^{n-1} A_{p+1}(k) g(k) v(k) - \sum_{k=p+1}^{n} A_{p+1}(k) g(k) v(k-1) } \\
& =& \sum_{k=p}^{n-1} v(k) \biggl( g(k) A_{p+1}(k) - g(k+1) A_{p+1}(k+1) \biggr) \\
&=& \sum_{k=p}^{n-1} g(k)  A_p(k) v(k)  \left( \alpha_{k-p} - \alpha_{k+1} \frac{g(k+1)}{g(k)} \right).
\end{eqnarray*}
Here, we use the facts that by definition $A_{p+1}(k) = A_p(k) \alpha_{k-p}$ and $A_{p+1}(k+1) = A_p(k) \alpha_{k+1}$.
We conclude by substituting for $\alpha_{k+1} g(k+1)$ using \eqref{eq:alphavals2}, and noting that $ \alpha_{k+1} + \alpha_{k-p} - 1 = \beta_{k+1} + \beta_{k-p}$. 
\end{proof}

To prove Lemma \ref{lem:QnpIdentity}, we observe that
equivalently, the family of polynomials \\ $(Q_{n,p}(X_1,\ldots,X_p))_{n \geq 0}$ can be defined using the generating function:
\begin{equation}\label{eq:QnpGenerating}
\sum_{n=0}^\infty Q_{n,p}(X_1,\ldots,X_p) t^n = \frac{1}{(1-tX_1)(1- t X_2) \ldots (1-tX_p)}
\end{equation}

\begin{proof}[Proof of Lemma \ref{lem:QnpIdentity}] We first notice by direct calculation that
\begin{eqnarray*}
\frac{1}{(1-tX_2) \ldots (1-tX_{p+1})}-\frac{1}{(1-tX_1) \ldots (1-tX_p)}
&=& \frac{ t(X_{p+1}-X_1)}{ (1-tX_1) \ldots (1-tX_{p+1})}.
\end{eqnarray*}  
Using equation~\eqref{eq:QnpGenerating}, we thus have 
\begin{eqnarray*}
\lefteqn{
\sum_{n=0}^\infty Q_{n,p}(X_2,\ldots,X_{p+1}) t^n-\sum_{n=0}^\infty Q_{n,p}(X_1,\ldots,X_p) t^n} \\
 &=& (X_{p+1}-X_1)\sum_{n=0}^\infty Q_{n,p+1}(X_1,\ldots,X_{p+1}) t^{n+1} \\
&=& (X_{p+1}-X_1)\sum_{n=1}^\infty Q_{n-1,p+1}(X_1,\ldots,X_{p+1}) t^{n},
\end{eqnarray*}
which yields equation~\eqref{eq:QnpIdentity} for $n \geq 1$. The case $n=0$ is obvious. \end{proof}

\section{Heuristics in continuous case}

We now explain some calculations  in the continuous case that helped us to find a rigorous proof of Theorem~\ref{thm:somon}, and that help suggest our conjecture about Renyi and Tsallis entropies. We remark that Ordentlich \cite{ordentlich} used the original paper of Shepp and Olkin \cite{shepp} to motivate conjectures concerning continuous random variables.

Let us consider a density function $f(x)$, defined for $x \in \re$, which is assumed to be everywhere positive, smooth and with all derivatives well-behaved at $\pm \infty$. This density will serve as a continuous analogue of both the mass functions $(f(k))_k$ and $(g(k))_k$. As a consequence, one could also see the function $\frac{1}{2}-\frac{\log(f)'(x)}{4}$ (resp. $-\frac{\log(f)'(x)}{4}$) as continuous analogues of the family $(\alpha)_k$ (resp. $(\beta)_k$). We will make the assumptions that $(\log f)'' \leq 0$ and $(\log f)''' \geq 0$, which correspond to the property $\alpha_k \leq \alpha_{k+1}$ and $\alpha_k -2 \alpha_{k+1}+\alpha_{k+2} \leq 0$. We now prove the continuous version of equation~\eqref{eq:oddpower} (in the case where $q=1$) and~\eqref{eq:TsallisBetaCondition} (for $q \neq1$):

\begin{proposition} Suppose that
$(\log f)'' \leq 0$ and $(\log f)''' \geq 0$ then for every real parameter $q>0$ and every integer $r \geq 1$ we have:
\begin{equation*}
\int_\re f(x)^q (\log f(x)')^{2r+1} dx \geq 0.
\end{equation*}
\end{proposition}

\begin{proof} For any $0 \leq p \leq r$ we set
\begin{equation*}
I_{r,p} = (-1)^p \frac{A_p}{q^p} \int_\re f(x)^q \left(\log f(x)' \right)^{2r-2p+1} \left(\log f(x)'' \right)^p dx,
\end{equation*}
with $A_0 = 1$ and $A_p = \prod_{k=0}^{p-1} (2r-2k)$ for $p \geq 1$. In particular $A_r = 0$ so $I_{r,r} = 0$. We now prove that the sequence $(I_{r,p})$ is non-increasing in $p$: for every $0 \leq p \leq r-1$,
using the fact that $f^q (\log f)' = f^q f'/f = f^{q-1} f' = (f^q)'/q$
 we have:
\begin{eqnarray*}
I_{r,p} &=& (-1)^p \frac{A_p}{q^{p+1}} \int_\re (f(x)^q)' (\log f(x)')^{2r-2p} (\log f(x)'')^p dx \\
&=& (-1)^{p+1} \frac{A_p}{q^{p+1}} (2r-2p) \int_\re f(x)^q  (\log f(x)')^{2r-2p-1} (\log f(x)'')^{p+1} dx \\
&& + (-1)^{p+1} \frac{A_p}{q^{p+1}} p \int_\re f(x)^q (\log f(x)')^{2r-2p} (\log f(x)'')^{p-1} \log f(x)''' dx.
\end{eqnarray*}
Here, again we apply integration by parts, followed by the product rule.
The first integral is exactly $I_{r,p+1}$, and the second one is non-negative because of the assumptions on $\log f''$ and $\log f'''$. We thus have $I_{r,p} \geq I_{r,p+1}$. We thus have
\begin{equation*}
\int_\re f(x)^q (\log f(x)')^{2r+1} dx = I_{r,0} \geq I_{r,r} = 0,
\end{equation*}
which proves the result.
\end{proof}


\end{document}